\documentclass{jT}

\usepackage[colorlinks=false, pdfstartview=FitV, linkcolor=blue,citecolor=black, urlcolor=blue]{hyperref}

\usepackage{amsmath,amsthm,amssymb,mathrsfs}
\usepackage{color}
\usepackage{hyperref}
\usepackage{url}
\usepackage{verbatim}

\newcommand{\bburl}[1]{\textcolor{blue}{\url{#1}}}

\newtheorem{thm}{Theorem}[section]

\newtheorem{cor}[thm]{Corollary}
\newtheorem{lem}[thm]{Lemma}
\newtheorem{prop}[thm]{Proposition}

\theoremstyle{definition}

\newtheorem{rem}[thm]{Remark}

\numberwithin{equation}{section}

\newcommand\be{\begin{equation}}
\newcommand\ee{\end{equation}}
\newcommand\ben{\begin{enumerate}}
\newcommand\een{\end{enumerate}}

\newcommand{\A}{\ensuremath{{\bf A}}}

\newcommand{\C}{\ensuremath{{\mathbb C}}}
\newcommand{\Z}{\ensuremath{{\mathbb Z}}}
\newcommand{\f}{\ensuremath{{\mathfrak f}}}
\newcommand{\GL}{\ensuremath{{\mathrm{GL}}}}
\newcommand{\n}{\ensuremath{{\mathfrak n}}}

\newcommand{\OO}{{\mathcal O}}

\newcommand{\SL}{{\mathrm{SL}}}

\numberwithin{equation}{section}

\begin{document}

\title{A Dedekind-Rademacher cocycle for Bianchi groups}

\author{Kim Klinger-Logan}
\address{Kim Klinger-Logan\\Kansas State University, 1228 N Martin Luther King Jr Dr, 138 Cardwell Hall, Manhattan, KS 66506, \\USA}
\email{kklingerlogan@ksu.edu}

\author{Kalani Thalagoda}
\address{Kalani Thalagoda\\Tulane University, 6823 St. Charles Avenue, New Orleans, LA 70118,\\USA}
\email{kthalagoda@tulane.edu}

\author{Tian An Wong}
\address{Tian An Wong\\University of Michigan-Dearborn, 4901 Evergreen Rd, 2002 CASL Building, Dearborn, MI 48128, \\USA}
\email{tiananw@umich.edu}




\maketitle

\begin{resume}
Nous construisons une g\'en\'eralisation du cocycle Dedekind-Rademacher aux sous-groupes de congruence de $\mathrm{SL}_2(\mathbb C)$, et en d\'eduisons certaines de ses propri\'et\'es de base. En particulier, nous montrons qu'il param\'etre une famille de valeurs $L$ et prouvons l'int\'egralit\'e de ces valeurs.
\end{resume}

\begin{abstr}
We construct a generalization of the Dedekind-Rademacher cocycle to congruence subgroups of $\mathrm{SL}_2(\mathbb C)$, and derive some of its basic properties. In particular, we show that it parametrizes a family of $L$-values and prove the integrality of these values.
\end{abstr}


\section{Introduction}

\subsection{The classical Dedekind-Rademacher homomorphism}

Consider the map $\phi: \mathrm{SL}_2(\Z)\to \Z$ given by
$$
\phi\begin{pmatrix}a & b \\ c & d\end{pmatrix}
= \begin{cases}
      \dfrac{a+d}{c}-12\dfrac{c}{|c|}\cdot S(a/|c|) & c\neq 0 \\
      \dfrac{b}{d} & c=0 
    \end{cases},
$$
    where 
$$
S(a/c) =\sum_{k=1}^{c-1} B_1\left(\frac{k}{c}\right)B_1\left(\frac{hk}{c}\right),\qquad (a,c)=1
$$
is the classical Dedekind sum, with $B_1(x):=x-\lfloor{x}\rfloor-1/2.$ The Dedekind-Rademacher homomorphism $\phi_N: \Gamma_0(N)\to\Z$ is then defined in \cite{rademacher} as the difference
\begin{align*}
\phi_N\begin{pmatrix}a & b \\ c & d\end{pmatrix}
:= \phi\begin{pmatrix}a & bN \\ c/N & d\end{pmatrix}-\phi\begin{pmatrix}a & b \\ c & d\end{pmatrix}.
\end{align*}
Mazur \cite{Maz} gave an equivalent definition as follows, 
\begin{align*}
\phi_N\begin{pmatrix}a & b \\ Nc & d\end{pmatrix}&= \begin{cases}
      \frac{(N-1)(a+d)}{Nc}-12\frac{c}{|c|}S^N(a/N|c|) & c\neq 0 \\
      \frac{(N-1)b}{d} & c=0 
    \end{cases} 
\end{align*}
where $S^N(x)=S(x)-S(Nx)$ for any $x\in \mathbb Q$.  

For $N=p$ a prime and $\Gamma = \mathrm{SL}_2(\Z[1/p])$,  Darmon, Pozzi, and Vonk showed that $\phi_p$ induces a rigid meromorphic cocycle on $\Gamma$, called the Dedekind-Rademacher cocycle \cite{DPV}. At certain real quadratic arguments the values at $\phi_p$ describe global $p$-units in the narrow Hilbert class field of the associated real quadratic field. 

\subsection{Generalizing to Bianchi groups}

In this note, we construct a generalization, $\Phi_N$, of the Dedekind-Rademacher homomorphism to Bianchi groups, specifically, congruence subgroups $\Gamma_0(\n)$ of $\mathrm{SL}_2(\mathbb C)$. We construct $\Phi_N$ more generally as the specialization of a 1-cocycle following work of Sczech \cite{Scz3} and Ito \cite{Ito} on elliptic Dedekind sums. 

In Section \ref{def}, we define the cocycle $\Phi_N$ on $\Gamma_0(\n)$ in \eqref{PhiNdef} and derive its basic properties, including the eigenvalues of certain Hecke operators acting on $\Phi_N$ in Proposition \ref{heck}. In Section \ref{Lfunsect}, we show as an application that it parametrizes a family of $L$-values in Theorem \ref{Lvalthm} and obtain as a result the integrality of these special values in Corollary \ref{Lint}. 


Our work is in part motivated by renewed interest in Eisenstein cocycles, Dedekind sums, and their applications to arithmetic. 
Besides the work of \cite{DPV}, Bergeron, Charollois, and Garcia recently constructed an $(n-1)$-homogeneous cocycle on $\Gamma_0(\mathfrak p)$, where $\mathfrak p$ is a prime ideal in $\mathcal O_K$ \cite{BCG}. Their construction can be viewed as a smoothed version of the Eisenstein cocycle of \cite{FKW}. In the case $n=2$, it is expected that $\Phi$ is cohomologous to the non-smoothed Eisenstein cocycle $\Psi_0$. In this context, it will be natural to expect that a smoothed form of $\Phi_N$ will be cohomologous to the cocycle of \cite{BCG}. We speculate on this connection further in Remark \ref{Wes} below.
Elsewhere, conjectures of Sharifi, supported by work of Fukaya and Kato, relate cup products in Galois cohomology to elements in the Eisenstein ideal. In \cite{FKS}, they pose the question of formulating the conjectures over imaginary quadratic fields. Recent work of Sharifi and Venkatesh \cite{SV} gave a motivic construction of Eisenstein cocycles over $\mathbb Q$, but its extension to imaginary quadratic fields is not immediately apparent. We hope that this work will motivate further study in the Bianchi case.

\section{The construction}
\label{def}
\subsection{Definitions}
Let $L=\omega_1\Z+\omega_2\Z$ be a lattice in $\mathbb{C}$ with $\text{Im}(\omega_1/\omega_2)>0$ and let $\mathcal{O}_L:=\{m\in\C~|~mL\subset L\}$ be the ring of multipliers. We assume that $\mathcal O_L$ is an order in an imaginary quadratic field $K$ such that the associated elliptic curve $\mathbb C/L$ has CM by $K$. When the context is clear, we shall sometimes write $\mathcal O = \mathcal O_L$ for convenience.

 For $k\in \mathbb{Z}_{\geq 0}$, define the Kronecker Eisenstein series as
$$
E_k(x)=\sum_{\substack{w\in L\\w\neq 0}} (w+x)^{-k}|w+x|^{-s}\big|_{s=0}, \qquad x,s\in \mathbb C
$$
where the value at $s=0$ is understood in the sense of the analytic continuation. Elliptic Dedekind sums are then defined as 
$$
D(a,c) = \frac{1}{c}\sum_{r\in L/cL} E_1\left(\frac{ar}{c} \right) E_1\left(\frac{r}{c} \right),
$$
where $a,c\in \mathcal O_L$ with $c\neq 0$. We shall furthermore assume that $\mathcal O_L$ is not isomorphic to $\Z[i]$ or $\Z[e^{2\pi i /3}]$, in which case $D(a,c)$ is trivial.
Sczech \cite{Scz3} defined a generalization of the map $\phi$ as a homomorphism $\Phi = \Phi_L: \mathrm{SL}_2(\mathcal{O}_L)\to \C$ where
$$
\Phi(A)
= \begin{cases}
     E_2(0)I\left(\dfrac{a+d}{c}\right)- {D}(a,c)& c\neq 0 \\
    E_2(0)I\left(\dfrac{b}{d}\right) & c=0 
    \end{cases},
$$
with $I(z):=z-\bar z$ and $A = \begin{pmatrix}a & b \\ c & d\end{pmatrix}\in \mathrm{SL}_2(\mathcal{O}_L)$. By \cite[Satz 4]{Scz3}, its image lies in the ring of integers of the number field \[F := \mathbb Q(g_2,g_3,\sqrt{D}),\] where $D$ is the discriminant of $\mathcal O_L$ and $g_2=g_2(L), g_3 = g_3(L)$ are algebraic integers associated to the elliptic curve defined by $L$.

More generally, let
\[
E(x) = \frac{2\pi i }{D(L)}\sum_{\substack{w\in L\\w\neq 0}}\frac{\overline{w+x}}{w+x}|w+x|^{-2s}\big|_{s=0}
\]
where $D(L) = w_1\bar w_2 - \bar w_1 w_2$ if $L = \mathbb Zw_1+\mathbb Z w_2$ and Im$(w_1/w_1)>0$, satisfying $E(0)=E_2(0)$.
For $p,q\in\C/L$ and $a,c\in\mathcal{O}_L$ we define 
\[
D(a,c;p,q) = \frac{1}{c}\sum_{r\in L/cL} E_1\left(\frac{a(r+p)}{c} +q\right) E_1\left(\frac{r+p}{c} \right),
\]
and 
\begin{align*}
\Phi(A)(p,q) = -&\overline{\left(\frac{a}{c}\right)}E(p) -\overline{\left(\frac{d}{c}\right)}E(p^*)\\ & - \frac{a}{c}E_0(p)E_2(q)-\frac{d}{c}E_0(p^*)E_2(q^*)- D(a,c;p,q)
\end{align*}
for $c\neq 0$, and
\[
\Phi(A)(p,q) = -\overline{\left(\frac{b}{d}\right)}E(p)-\frac{b}{d}E_0(p)E_2(q)
\]
for $c=0$, where
\[
(p^*,q^*)=(p,q)A=(ap+cq,bp+dq).
\]
This map $\Phi(A)(p,q) $ satisfies the cocycle relation 
\be
\label{coc}
\Phi(AB)(p,q) = \Phi(A)(p,q)+\Phi(B)((p,q)A).
\ee
That is, if we let $\mathscr F$ be the space of functions from $(\C/L)^2$ to $\C$, then $\Phi(A)(p,q)$ represents a nontrivial cohomology class in $H^1(\SL_2(\mathcal O_L),\mathscr F)$, where the action of $\SL_2(\mathcal O_L)$ on $\mathscr F$ is given by $(Af)(x) = f(xA)$.

Ito \cite[Theorem 1]{Ito} later constructed a function $H(u;p,q)$ on the product of the upper-half space $\mathbb{H}^3$ and $\C^2$ such that 
\[
\Phi(A)(p,q) = H(Au;p,q)-H(u;(p,q)A)
\]
for all $u\in \mathbb H^3$ and $A\in \mathrm{SL}_2(\mathcal{O}_L)$. In particular, the special value $H(u):=H(u;0,0)$ satisfies
\[
\Phi(A) = H(Au)-H(u), 
\]
analogous to the imaginary part of the logarithm of the Dedekind $\eta$ function.

\subsection{The construction of $\Phi_N$}

Let $\n\subset \mathcal O_L$ be a nontrivial integral ideal generated by $N\in \OO_L$, and assume from now on that $N|L$, i.e., that $N$ divides the conductor of $\mathcal O_L$ as an order in $K$. 
Define the congruence subgroup
\[
\Gamma_0(\n) = \left\{\begin{pmatrix}a & b \\ c & d\end{pmatrix} \in \SL_2(\OO_L): c \in \n\right\},
\] 
and let
\[
D^N(x,y):= D(Nx,y) - D(x,y).
\]
 Define the Dedekind-Rademacher-Bianchi homomorphism for congruence subgroups $\Gamma_0(\n) \subset \mathrm{SL}_2(\C)$ as 
        \begin{align*}
        \Phi_N\begin{pmatrix}a & b \\ c & d\end{pmatrix}
        &:= \Phi\begin{pmatrix}a & bN \\ c/N & d\end{pmatrix}-\Phi\begin{pmatrix}a & b \\ c & d\end{pmatrix}\\
    &= \begin{cases}
     E_2(0)I\left(\frac{(N-1)(a+d)}{c}\right)- D^N(a,c)& c\neq 0 \\
    E_2(0)I\left(\frac{(N-1)b}{d}\right) & c=0.
    \end{cases}
\end{align*}


It is straightforward to extend it to a function of $(p,q)$ for $p,q\in\mathbb{C}/L$. We define more generally  
\[
D^N(x,y;p,q):= D(Nx,y;p,q) - D(x,y;p,q)
\]
and
\be
\label{PhiNdef}
\Phi_N\begin{pmatrix}a & b \\ c & d\end{pmatrix}(p,q)
:= \Phi\begin{pmatrix}a & bN \\ c/N & d\end{pmatrix}(p,q)-\Phi\begin{pmatrix}a & b \\ c & d\end{pmatrix}(p,q)
\ee
for $p,q\in \C/L$, and we denote 
\[
\Phi^{p,q}_N(A) = \Phi_N(A)(p,q),\qquad \Phi_N(A)=\Phi_N^{(0,0)}(A)=\Phi_N(A)(0,0).
\]
More explicitly, for $c\neq 0$ and $p,q\in \mathbb{C}/L$, we have
\begin{align*}
\Phi_N\begin{pmatrix}a & b \\ c & d\end{pmatrix}(p,q)&
=-(N-1)\overline{\left(\frac{a}{c}\right)}E(p)-(N-1)\overline{\left(\frac{d}{c}\right)}E(p^*)\\ &-(N-1)\frac{a}{c}E_0(p)E_2(q)-(N-1)\frac{d}{c}E_0(p^*)E_2(q^*)\\
& -D^N(a,c;p,q),
\end{align*} 
where $$(p^*,q^*)=(p,q)A_N=(ap+cq/N,Nbp+dq),$$
and for $c=0$, 
\[
\Phi_N(A)(p,q) = -(N-1)\overline{\left(\frac{b}{d}\right)}E(p)-(N-1)\frac{b}{d}E_0(p)E_2(q).
\]

\begin{prop}
$\Phi_N(A)(p,q)$ is a $1$-cocycle on $\Gamma_0(\n)$ with values in $\mathscr F$.
\end{prop}

\begin{proof}
In $\Gamma_0(\n)$, let
\[
\begin{pmatrix}a & b \\ c & d\end{pmatrix} \begin{pmatrix}a' & b' \\ c' & d'\end{pmatrix}
= \begin{pmatrix}a'' & b'' \\ c'' & d''\end{pmatrix},
\]
and note that 
\[
\begin{pmatrix}a & bN \\ c/N & d\end{pmatrix} \begin{pmatrix}a' & b'N \\ c'/N & d'\end{pmatrix}=\begin{pmatrix}a'' & b''N \\ c''/N & d''\end{pmatrix}.
\]
Let us denote these equations by $AA' = A''$ and $A_NA'_N=A''_N$ respectively.

Then using the cocycle property \eqref{coc}, we have by direct computation for $p,q\in \C/L$,
\begin{align*}
\Phi_N(A'')(p,q) & =  \Phi(A''_N)(p,q)- \Phi(A'')(p,q) \\
 & = \Phi(A_N)(p,q)+  \Phi(A'_N)((p,q)A_N)  - \Phi(A)(p,q) \\
 & \ \ \ \ \  -\Phi(A')((p,q)A)\\
 & =   \Phi_N(A)(p,q) + \Phi_N(A')((p,q)A),
\end{align*}
where we have used the property that $(p,q)A_N \equiv (p,q)A$ mod $L^2$. 
\end{proof}

\begin{cor}
\label{int}
The map $\Phi_N$ is a homomorphism with values in $\mathcal O_F$ for $F = \mathbb Q(g_2,g_3,\sqrt{D})$.
\end{cor}

\begin{proof}
This follows from specializing the preceding proof to $(p,q)=(0,0)$. Namely, we have
\[
\Phi_N\begin{pmatrix}a'' & b'' \\ c'' & d''\end{pmatrix} = \Phi_N\begin{pmatrix}a& b \\ c& d\end{pmatrix}+\Phi_N\begin{pmatrix}a' & b' \\ c' & d'\end{pmatrix}.
\]
The integrality of $\Phi_N$ follows directly from that of $\Phi$.
\end{proof}

Let $H$ be as defined in \cite[(4)]{Ito} with Fourier expansion
\begin{equation}\label{eq:H}H(u) = E_2(0)(z-\bar z) -\frac{4\pi}{D(L)}\, v \sum'_{m\in L, n\in L'} \frac{\bar m n }{|mn|} K_1(4\pi |mn|v)e(mnz).
\end{equation}
\begin{cor}\label{cor:trans}
Define $H_N(u):= H(Nu)-H(u)$ for $H$ as defined in \eqref{eq:H}. Then $\Phi_N$ satisfies the transformation law 
\[
\Phi_N(A)= H_N(Au)- H_N(u).
\]
\end{cor}

\begin{proof}
Since $\Phi$ is independent of choice of $u\in\mathbb{H}^3$,
we can write $\Phi(A_N)= H(A_NNu) - H(Nu)$. Letting $A_N=\begin{pmatrix}a & bN \\ c/N & d\end{pmatrix}$, so that 
\[
Au=\frac{a u + b}{cu+d},\qquad A_NNu= \frac{aNu + Nb}{\frac{c}{N}Nu+d}, 
\] it follows that 
\begin{align*}
\Phi_N(A)&= \Phi(A_N) - \Phi(A)  = H(A_NNu) - H(Nu)- H(Au)+H(u)\\
&= H(NAu) - H(Au)-(H(Nu)-H(u))=H_N(Au) - H_N(u)
\end{align*}
and the result follows.
\end{proof}

\noindent We then have the following uniqueness theorem analogous to \cite[Theorem 4]{Ito}.
To state the theorem fully, fix a sufficiently large real number $V$ and define for every parallelogram $P$ in $\mathbb{C}$ the set \[
\Pi (P)= \{ z+jv\in \mathbb{H}~|~ z\in P, v>V\}
\]
and fix the usual $\SL_2(\mathbb C)$-invariant measure
$
v^{-3} {dx\ dy\ dv}
$
on $\mathbb H^3$. The discrete group $\Gamma_0(\n)$ acts discontinuously on $\mathbb H^3$ and the volume of any fundamental domain of the action with respect to this measure is finite.

\begin{prop}Assume $\mathcal{O}_L\neq\mathbb{Z}, \mathbb{Z}[i], \mathbb{Z}[e^{2\pi i /3}]$. Then the function $H_N(u)$ is not identically zero. Moreover, if a complex-valued $C^2$-function $h_N$ on $\mathbb{H}$ satisfies the following conditions, then it coincides with $H_N(u)$: 
\begin{enumerate}
\item For any matrix $M\in \Gamma_0(\n)$ and any parallelogram $P$, $h_N$ is bounded on $M(\Pi(P)),$
\item $\Delta h_N=0$ with $\Delta$ defined as $\Delta= v^2\left(\partial_x^2+\partial_y^2+\partial_v^2\right)-v\partial_v$,
\item $h_N(Au)=h_N(u)+\Phi_N(A)$ for all $A\in \Gamma_0(\n)$,
\item $h_N(-z+jv)=-h_N(z+jv)$ for $u=z+jv\in\mathbb{H}^3$.
\end{enumerate}
\end{prop}
\begin{proof}
The function $H_N$ satisfies the first and third conditions by \eqref{eq:H} and Corollary \ref{cor:trans} respectively. The other two conditions follow from the fact that $H(u)$ is harmonic and satisfies $H(-z+jv) = -H(z+jv)$. On the other hand, given $h_N$ satisfying (1), (2), and (3), let $f_N = h_N - H_N$. It is known that there exists a fundamental domain $\mathcal F$ for $\SL_2(\mathcal O_L)\backslash \mathbb H^3$ of the form
\[
\mathcal F = \mathcal F_0\cup \left(\bigcup_{i=1}^n M_i(\Pi(P_i))\right), \qquad M_i \in \SL_2(\mathbb C),
\] 
where $\mathcal F_0$ is a compact set and $P_i$ are parallelograms in $\mathbb C$ for $i=1,\dots,n$. Since $\Gamma_0(\n)$ has finite index in $\SL_2(\mathcal O_L)$ we can also obtain a similar fundamental domain $\mathcal F_\n$ for $\Gamma_0(\n)\backslash \mathbb H^3$. Then $|f_N|^2$ is bounded and integrable on $\mathcal F_\n$ and \cite[Satz 5]{roel} show that
\[
0 = - \int_{\mathcal F_\n}\bar f_N\Delta f_N\frac{dx\ dy\ dv}{v^3} = \int_{\mathcal F_\n} |\partial_x f_N|^2 + |\partial_y f_N|^2 + |\partial_v f_N|^2 \frac{dx\ dy\ dv}{v^3}.
\]
Hence $\partial_xf_N=\partial_yf_N=\partial_vf_N =0$ on $\mathcal F_\n$, and $f_N$ must be constant on $\mathbb H^3$. Since $h_N$ satisfies (4), then $f_N=0$.
\end{proof}

\begin{rem} In the classical case, if we define the quotient of eta functions $h_N(z)= \eta(Nz)/\eta(z)$, then the Dedekind-Rademacher homomorphism $\phi_N$ can be interpreted as the period of $\log h_N$ by the rule
\[
\log h_N(\gamma z) - \log h_N(z) = \frac{\pi i }{12}\phi_N(\gamma), \qquad \gamma=\begin{pmatrix}a & b \\ c & d\end{pmatrix}
\]
for any $z\in\mathbb{H}^2$ \cite[(2.2)]{Maz}. Then just as $H(u)$ is the analogue of the $\log\eta(z)$, so $H_N(u)$ can be viewed as the analogue of $\log h_N(z)$.
\end{rem}

\subsection{Hecke action} 

Let us first recall the Hecke operators in our setting. Let $\Gamma$ be a finite index subgroup of the Bianchi group $\mathrm{PSL}_2(\mathcal{O}_K)$. 
Given $g\in \mathrm{PSL}_2(K)$, consider the 3-folds $M,M_g, M^g$ associated to 
$$
\Gamma, \ \ \ \Gamma_g=\Gamma\cap g\Gamma g^{-1}, \ \ \  \Gamma^g =\Gamma\cap g^{-1}\Gamma g
$$
respectively. We have finite coverings induced by inclusion of fundamental groups
$
r_g: M_g\to M$ and $r^g:M^g\to M
$,
and so an isometry  $\tau: M_g\to M^g$ induced by the isomorphism between $\Gamma_g$ and $\Gamma^g$ given by conjugation by $g$. The composition $s_g:=r^g\circ \tau$ gives us a second finite covering from $M_g$ to $M$. The coverings $r_g$ induce linear maps between homology groups
$$r_g^*: H^i(M,\mathcal{V})\to  H^i(M_g,\mathcal{V}).$$
The process of summing finitely many preimages in $M_g$ of a point of $M$ under $s_g$ leads to 
$$
s_g^*: H^i(M_g,\mathcal{V})\to  H^i(M,\mathcal{V}).
$$
Note that $s_g^*$ is equivalent to the composition
$
H^i(M_g,\mathcal{V})\to   H^i(M^g,\mathcal{V})\to  H^i(M,\mathcal{V}),
$
where the second arrow  is the corestriction map. 

The Hecke operator $T_g$ associated to $g\in \mathrm{PSL}_2(K)$ is then defined as the composition 
$$
T_g:= s_g^*\circ r_g^*: H^i(M,\mathcal{V})\to H^i(M,\mathcal{V}).
$$
Moverover, if we take $g\in \mathrm{GL}_2(K)$ and let 
\[
g=\begin{pmatrix} 1& 0 \\ 0 & p\end{pmatrix},
\]
where $p$ is a prime in $\mathcal{O}$ and $p\mathcal{O}$ is a prime ideal in $\mathcal{O}$.  Denote $T_p$ by the associated Hecke operator. We recall the following properties of the homomorphism $\Phi$ proved in \cite[Satz 7]{Scz3}.

\begin{lem} For each homomorphism $\Phi=\Phi_L$ where $\mathcal{O}L\subset L$ we have 
\[
\Phi\circ T_p=(p+\bar p)\Phi.
\]
If $x=\begin{pmatrix} 1& 0 \\ 0 & -1\end{pmatrix}$,  this gives $\Phi\circ T_x = -\Phi.$ 
\end{lem}

It is then straightforward to show the same for the new homomorphism $\Phi_N$.

\begin{prop}\label{heck}
Assume $\mathcal{O}L\subset L$. For each $N|L$  we have 
$$
\Phi_N\circ T_p =(p+\bar p)\Phi_N.
$$
For $x=\begin{pmatrix} 1& 0 \\ 0 & -1\end{pmatrix}$,  then $\Phi_N\circ T_x = -\Phi_N.$
\end{prop}
\begin{proof} 
Applying the preceding lemma, the definition of $\Phi_N$ gives 
\begin{align*}\Phi_N\begin{pmatrix}a & b \\ c & d\end{pmatrix}\Big|T_p
        &= \Phi\begin{pmatrix}a & bN \\ c/N & d\end{pmatrix}\Big|T_p-\Phi\begin{pmatrix}a & b \\ c & d\end{pmatrix}\Big|T_p\\
        &= (p+\bar p)\Phi\begin{pmatrix}a & bN \\ c/N & d\end{pmatrix}-(p+\bar p)\Phi\begin{pmatrix}a & b \\ c & d\end{pmatrix}\\
        & = (p+\bar p)\Phi_N\begin{pmatrix}a & b \\ c & d\end{pmatrix}.
        \end{align*}
The second assertion follows similarly.
\end{proof}

\begin{rem}
As in the classical case, it would be desirable to show that $\Phi_N^{p,q}$ is Eisenstein in the sense that it lies in the Eisenstein cohomology of $\SL_2(\mathcal O_L)$. Unfortunately, we have not yet been able to verify this. 
\end{rem}

\section{Special values of $L$-functions}
\label{Lfunsect}

Define $u=z+jv$ for $z\in\mathbb{C}$ and $v\in\mathbb{R}_{>0}$ and note that $\SL_2(\mathbb{C})$ acts on $\mathbb{H}^3$ by $Mu=(au+b)(cu+d)^{-1}$. If an element of $\SL_2(\mathcal{O})$ has a fixed point in $K$, it must be conjugate in $\Gamma_0(\mathfrak{n})$ to an element 
$\begin{pmatrix}\zeta & a \\ 0 &\zeta^{-1}\end{pmatrix}$ for $a\in K$ and $\zeta$ a root of unity in $K$. Thus for 
\begin{equation}\label{eq:assume}A_N=\begin{pmatrix}a& Nb \\ c/N &d\end{pmatrix}\in \Gamma_0(\mathfrak{n})\text{ with }(a+d)^2\neq 0, 1, 4,\end{equation} 
$A_N$
 has two distinct fixed points which we will denote $N\alpha$ and $N\alpha'$ in $\mathbb{C}$ as they are multiples of the fixed points $\alpha$ and $\alpha'$ for $\begin{pmatrix}a& b \\ c &d\end{pmatrix}\in \SL_2(\mathcal{O})$. The element $\varepsilon:=c\alpha +d$ is a unit in the field $K(\alpha)$ and in $K(N\alpha)$. We have $|\varepsilon|\neq 1$ and $\varepsilon\cdot\varepsilon' =1$ for $\varepsilon'=c\alpha'+d$. Now define 
\[Q_N(m,n) := (mN\alpha+n)(mN\alpha'+n).\] 
From \cite[p.\,159]{Ito}, we have 
\[Q_N((m,n)A_N) = Q_N(m,n) .\] 

For $p,q\in\mathbb{C}$, assume 
\begin{align}\label{eq:assume1}(p,q)A&\equiv (p,q)\, \text{mod}\, L^2 \ \ \text{and}\\
\label{eq:assume2}(p,q)A_N&\equiv (p,q)\, \text{mod}\, L^2.\end{align}
The matrix $A_N$ acts naturally on $(L+p)\times (L+q)$ and we can define the series
\begin{equation}
L_{N}(A_N,s;p,q) = \sum_{(m,n)}\frac{\overline{Q_N(m,n)}}{|Q_N(m,n)|^{2s}}
\end{equation}
for $\text{Re}(s)>3/2$ where the sum over $(m,n)$ runs through the complete set of elements of $(L+p)\times (L+q) -\{(0,0)\}$ not associated with respect to the action of $A_N$.

The following result and proof follows similarly to that of  \cite[Theorem 3]{Ito}.

{\thm 
\label{Lvalthm}
Let $\theta = 1$ if $|\varepsilon|>1$ and $\theta=-1$ otherwise. Given the notation and assumptions \eqref{eq:assume} -- \eqref{eq:assume2}, we have 

\begin{align*}
\Phi_N(A)(p,q)&=\theta(\alpha -\alpha')\Big(NL_{N}(A_N,1;p,q)  - L(A,1;p,q) \Big)\\
&= -(N-1)\overline{\left(\frac{a+d}{c}\right)}E_2(p) - (N-1)\frac{a+d}{c}E_0(p)E_2(q) 
\\ & \ \ \ \ \ \ - D^N(a,c;p,q) \nonumber
\end{align*}
and
\begin{align*}
	\theta(\alpha -\alpha')L_{N}(A_N,1;p,q)&=-\overline{\left(\frac{a+d}{c}\right)}E_2(p) - \frac{a+d}{c}E_0(p)E_2(q) \\
	& \ \ \ \ \ \ \   - \frac{1}{N} D(a,c/N;p,q).
\end{align*}
}

We note that if we take $\mathfrak n$ to be a prime ideal, the difference of $L$-functions on the righthand side can be viewed as a smoothing of the $L$-function at $\mathfrak n$.

\begin{proof} Let $L_K=\{km~|~k\in K,m\in L\}$.  Note that every element $\mu_N$ of $L_KN\alpha +L_K$
is uniquely expressed as $\mu_N=mN\alpha+n$ with $m,n\in L_K$. Let $\mu'_N = mN\alpha'+n$ and define $Q_N(\mu_N) =\mu_N\mu_N'$. Assumptions \eqref{eq:assume} -- \eqref{eq:assume2} imply 
\begin{align*}\mathcal{M}_{p,q}&:= (L+p)\alpha+L+q\subset L_K \alpha +L_K,\\
\mathcal{M}^N_{p,q}&:= (L+p)N\alpha+L+q\subset L_K N\alpha +L_K,
\end{align*}
also $\varepsilon\mathcal{M}_{p,q}=\mathcal{M}_{p,q}$, and  $\varepsilon\mathcal{M}^N_{p,q}=\mathcal{M}^N_{p,q}$.  We have that, for $p,q\in\mathbb{C}/L$,
\[L_{N}(A_N,s;p,q) = \sum''_{\mu_N}\frac{\overline{Q_N(\mu_N)}}{|Q_N(\mu_N)|^{2s}},\quad \text{Re}(s)>3/2\]
where the sum $ \displaystyle\sum''_{\mu_N}$ runs over a complete set of elements of $\mathcal{M}^N_{p,q}-\{0\}$ not associated with respect to $\varepsilon$.

Let 
\begin{align*}
u_N(t) &= z_N(t) +jv_N(t)\\ & = \frac{N\alpha t^2 +N\alpha '}{t^2+1} +j\frac{N|\alpha - \alpha'|t}{t^2+1}\\ &  =N z(t) +jN v(t)
=Nu(t)
\end{align*}
 for $u(t), z(t)$ and $v(t)$. 
 Note that $u_N(t)$ is a geodesics connecting $N\alpha$ and $N\alpha'$ and using the analogous identity in \cite[p.\,160]{Ito}, we have 
 \[A_N u_N(t) =u_N(|\varepsilon|^2t). \]

Denote the path from $u_N(1)$ to $u_N(|\varepsilon|^2 t)$ on this geodesic by $Z_{A_N}$. By Corollary \ref{cor:trans} and Theorem 1 of \cite{Ito},
\begin{equation*}
\int_{Z_{A_N}}\omega(p,q; s) - \int_{Z_{A}}\omega (p,q; s)= \Phi(A_N)(p,q)- \Phi(A)(p,q).
\end{equation*}
Both $\int_{Z_{A_N}}\omega(p,q; s)$ and $ \int_{Z_{A}}\omega (p,q; s)$ are holomorphic in $s\in\mathbb{C}$. The second integral $ \int_{Z_{A}}\omega (p,q; s)$ has been explicitly calculated in \cite[Theorem 1]{Ito}. We  provide the computation for the first along similar lines.

 For $\text{Re}(s)>1$, $\int_{Z_{A_N}}\omega(p,q; s)$  can be evaluated as follows.
First note that
 \[
\left(\frac{dz(t)}{dt}, \frac{dv(t)}{dt},\frac{d\bar z(t)}{t}\right) = (1+t^2)^{-2}\left(2(\alpha -\alpha ') t , |\alpha-\alpha'|(1-t^2), 2\overline{(\alpha - \alpha ')} t \right)
\]
\begin{align*}
&	\left(\frac{dz_N(t)}{dt}, \frac{dv_N(t)}{dt},\frac{d_N\bar z(t)}{t}\right) \\  & \ \ \ \ \ \ \ \ \ \ \ \ \ \ \ \ \ \ \ \ \ \ \ \ = N (1+t^2)^{-2}\left(2(\alpha -\alpha ') t , |\alpha-\alpha'|(1-t^2), 2\overline{(\alpha - \alpha ')} t \right).
\end{align*}

We then have
\begin{align}
\int_{Z_{A_N}}&\omega(p,q; s) \nonumber
\\
&= 
N\int_1^{|\varepsilon|^2}[E_z(u_N(t), s; p,q) 2(\alpha -\alpha ') t+  E_v(u_N(t), s; p,q) |\alpha-\alpha'|(1-t^2)\nonumber\\
 & \ \ \ \ \ \ \ \ \ \ \ \ \ \ \ \ \ \ \ \ +E_{\bar z}(u_N(t), s; p,q) 2\overline{(\alpha -\alpha ')} t](1+t^2)^{-2}\, dt \nonumber\\
&= 
 2N\sum'_{\substack{m\in L+p\\ n \in L+q}}\int_1^{|\varepsilon|^2}\Big[(\overline{mNz(t)+n})^2(\alpha -\alpha ') t \nonumber\\& \ \ \ \ \ \ \ \ \ \ \  \  +  (\overline{mNz(t)+n})\bar m Nv(t) |\alpha-\alpha'|(1-t^2) - (\bar m Nv(t))^2(\overline{\alpha -\alpha '}) t\Big]
 \nonumber\\
 & \ \ \ \ \ \ \ \ \ \ \ \ \ \ \ \ \ \ \ \ \ \ \ \ \ \ \times(|mNz(t) + n |^2 + |m Nv(t)|^2)^{-s-2} \frac{N^sv(t)^s}{{(1+t^2)^{2}}}\, dt. \nonumber\\\label{eq:inta_N1}
\end{align}
Let $\mu=m\alpha+n,\mu_N=mN\alpha+n\in \mathcal{M}_{p,q}$. Then 
\begin{align*}
&mz_N(t) +n = \frac{\mu_N t^2+\mu_N'}{t^2+1},\qquad mv_N(t) = \frac{|\alpha - \alpha'|  (\mu_N-\mu_N')t}{(\alpha -\alpha')(t^2+1)}, \text{ and }\\
 |&mz_N(t)+n|^2+|mv_N(t)|^2 = \frac{|\mu_N|^2t^2+|\mu_N'|^2}{t^2+1},
\end{align*}
and it follows that
\begin{align} &(\overline{mz_N(t)+n})^2(\alpha -\alpha ') t+  (\overline{mz_N(t)+n})\bar m v_N(t) |\alpha-\alpha'|(1-t^2) \nonumber
\\ \nonumber
&\ \ \ \ \ \ \ \ \ \ \ \ \ \ \ \ \ \ \ \ \ \ \  \ \ \ \ \ \ \ \ \ \ \ \ \ \ \ \ \ \ \ \ - (\bar m v_N(t))^2(\overline{\alpha -\alpha '}) t 
 = (\alpha -\alpha') \overline{Q_N(\mu_N)}t.\\\end{align}
 Note that similar identities also hold for $\mu$ and $\mu'$ from \cite[p.\,160]{Ito}.
Equation \eqref{eq:inta_N1} then becomes
\begin{align}
&2N\sum'_{\substack{\mu_N\in \mathcal{M}_{p,q}}}\int_1^{|\varepsilon|^2}\Big[ (\alpha -\alpha') \overline{Q_N(\mu_N)}t\Big]\left(\frac{|\mu_N|^2t^2+|\mu_N'|^2}{t^2+1}\right)^{-s-2} \frac{N^sv(t)^s}{{(1+t^2)^{2}}}\, dt\nonumber\\
 &= 
2N^{s+1}(\alpha -\alpha')\sum'_{\substack{\mu_N\in \mathcal{M}_{p,q}}}\int_1^{|\varepsilon|^2}  \overline{Q_N(\mu_N)} \frac{|\alpha - \alpha'|^st^s}{(t^2+1)^s}
\frac{t\,(1+t^2)^{s}}{{\left(|\mu_N|^2t^2+|\mu_N'|^2\right)^{s+2}}}\, dt\nonumber\\
 &= 
2N^{s+1}(\alpha -\alpha')|\alpha - \alpha'|^s\sum'_{\substack{\mu_N\in \mathcal{M}^N_{p,q}}}\int_1^{|\varepsilon|^2}  \overline{Q_N(\mu_N)} 
\left(\frac{t}{|\mu_N|^2t^2+|\mu_N'|^2}\right)^{s+2}\, \frac{dt}{t}. \nonumber\\\label{eq:ina_N2}
\end{align}
By change of variables $t\mapsto \frac{|\mu_N'|}{|\mu_N|}t$, \eqref{eq:ina_N2} becomes
\begin{align}
&2N^{s+1}(\alpha -\alpha')|\alpha - \alpha'|^s\sum'_{\substack{\mu_N\in \mathcal{M}^N_{p,q}}} \frac{\overline{Q_N(\mu_N)}}{|Q_N(\mu_N)|^{s+2}} \int^{\left|\frac{(\mu_N\varepsilon)'}{\mu_N\varepsilon}\right|}_{\left|\frac{\mu_N'}{\mu_N}\right|}
\left(\frac{t}{t^2+1}\right)^{s+2}\, \frac{dt}{t}\nonumber \\
  &\ \ \ = 
2N^{s+1}\theta(\alpha -\alpha')|\alpha - \alpha'|^s\sum''_{\substack{\mu_N\in \mathcal{M}^N_{p,q}}} \frac{\overline{Q_N(\mu_N)}}{|Q_N(\mu_N)|^{s+2}} \int^{\infty}_{0}
\left(\frac{t}{t^2+1}\right)^{s+2}\, \frac{dt}{t}\label{eq:inta_N3} \nonumber\\
\end{align}
where, as above, the sum $ \displaystyle\sum''_{\mu_N}$ runs over a complete set of elements of $\mathcal{M}^N_{p,q}-\{0\}$ not associated with respect to $\varepsilon$.
Thus, from \eqref{eq:inta_N3} and Theorem 1 of \cite{Ito}, we have

\begin{align}\label{eq:L} & \int_{Z_{A_N}}\omega(p,q; s) - \int_{Z_{A}}\omega (p,q; s)= \frac{\theta(\alpha -\alpha')|\alpha - \alpha'|^s}{\Gamma(s+2)}\Gamma\left(\frac{s+2}{2}\right)^2 \nonumber \\  & \hspace{3cm} \times \Big(N^{s+1}L_{N}(A_N,\frac{s+2}{2};p,q)  - L(A,\frac{s+2}{2};p,q) \Big). \nonumber\\ 
\end{align}

Then by analytic continuation to $s=0$, this becomes
\begin{align*}
\theta(\alpha -\alpha')\Big(NL_{N}(A_N,1;p,q)  - L(A,1;p,q) \Big).
\end{align*}
This gives the first identity in the theorem.

On the other hand, 
we have 
\begin{align*} \int_{Z_{A_N}}&\omega(p,q; s) - \int_{Z_{A}}\omega (p,q; s)  = \Phi(A_N)(p,q)- \Phi(A)(p,q) \end{align*} and 
\begin{align*} \int_{Z_{A}}\omega (p,q; s)  = \Phi(A)(p,q) .\end{align*}
From \eqref{eq:assume1} and \eqref{eq:assume2}, we have also that
\begin{align*}
\Phi(A_N)(p,q)&= -\overline{N\left(\frac{a+d}{c}\right)}E_2(p) - N\frac{a+d}{c}E_0(p)E_2(q) -  D(a,c/N;p,q),
\end{align*} 
which gives the second required formula.
\end{proof}

We note that equation \eqref{eq:L} gives the analytic continuation of the L-function
$L_N(A_N,s;p,q)$ to all of $s\in\mathbb{C}$. The following is an immediate consequence of Theorem \ref{Lvalthm} together with Corollary \ref{int}.
\begin{cor}
\label{Lint}
We have $(\alpha-\alpha')L_{N}(A_N,1;0,0)\in \mathcal O_F$.
\end{cor}

\begin{rem}\label{Wes}
We conclude with a brief comparison with a construction due to Weselmann. By Remark 2 of \cite[\S5]{Wes}, Sczech's cocycle can be realized as the specialization of the Eisenstein differential form $\omega_\text{Eis}$,
\[
\Phi(A)(u,v)=
2\pi \cdot\mathscr G_\infty \omega_\text{Eis}\big(\text{ch}_{\hat{\mathcal O}+u}\otimes\text{ch}_{\hat{\mathcal O}+v}\big)\left(A,1_f\right),
\]
for all $A\in \mathrm{SL}_2(\mathcal O_F)$ and $u,v\in F/\mathcal O$. Here ch$_X$ is the characteristic function of any subset $X$ of the finite adeles $\A_f$ of $F$, and $\mathscr G_\infty$ induces an isomorphism between de Rham cohomology and group cohomology. But Weselmann's construction relies on a choice of maximal compact subgroup $K$ of $\GL_2(F)$, and by allowing ramification at finite primes, one can deduce a similar 1-cocycle on $\Gamma_0(N)$. But it is not clear whether the cocycle obtained this way is related to ours, though it would be natural to expect that both occur in the associated Eisenstein cohomology class.
\end{rem}

\subsection*{Acknowledgments} This work was begun at the Rethinking Number Theory 2 workshop. K.-L. was supported by NSF grant number DMS-2001909 and W. was supported by NSF grant number DMS-2212924. 

\bibliography{BDR.bib}
\bibliographystyle{abbrv}
\end{document}